\documentclass{article}

\usepackage{geometry}
\usepackage{amsmath,amssymb,amsthm}

\newtheorem{theorem}{Theorem}
\newtheorem{lemma}{Lemma}
\newtheorem{definition}{Definition}
\newtheorem{claim}{Claim}
\newtheorem{remark}{Remark}

\def\R{\mathbb{R}}
\def\N{\mathbb{N}}
\def\S{\mathcal{S}}

\title{Strong spatial mixing of $q$-colorings on Bethe lattices}
\author{Qi Ge\thanks{ Department of Computer Science, University
of Rochester, Rochester, NY 14627.  Email: \{qge,stefanko\}@cs.rochester.edu.
Research supported, in part, by NSF grant CCF-0910415.} \and Daniel
\v{S}tefankovi\v{c}$^*$}

\begin{document}

\maketitle

\begin{abstract}
We investigate the problem of strong spatial mixing of $q$-colorings on Bethe lattices.
By analyzing the sum-product algorithm we establish the strong spatial mixing of $q$-colorings on $(b+1)$-regular
Bethe lattices, for $q \geq 1+\lceil 1.764b \rceil$. 
We also establish the strong spatial mixing of $q$-colorings on binary trees, for $q=4$.
\end{abstract}

\section{Introduction}

A $q$-coloring of a graph $G=(V,E)$  is a function $\sigma: V \to
[q]$ such that no edge is monochromatic (that is, for $\{u,v\}\in E$ we
have $\sigma(u) \neq \sigma(v)$). A measure $p$ on the set of
$q$-colorings of an infinite graph $G$ is an {\em infinite-volume
Gibbs measure} if for every finite region $R$, and for any
$q$-coloring $\sigma$ of $G$, the conditional probability
distribution $p(\cdot\,|\,\sigma(V \setminus R))$ is the uniform
distribution on $q$-colorings of $R$. It is known that there is at
least one infinite-volume Gibbs measure for any graph $G$. One
problem of interest in statistical physics
(c.f.~\cite{MR1746301}) is whether an infinite-volume Gibbs
measure has {\em strong spatial mixing}.

Given a $q$-coloring $\sigma$ and a set of vertices $U \subseteq V$, let $\sigma_U$ be the $q$-coloring
restricted to $U$. Given a measure $p$, a vertex $v \not\in U$, and a (partial) $q$-coloring $\sigma_U$, let $p_v^{\sigma_U}$ be
the marginal distribution on the colors of $v$ conditioned on $\sigma_U$. Let $\mathrm{dist}(u,v)$ be
the distance between $u,v$ in $G$, and let $\mathrm{dist}(v,U)=\min_{u\in U}\mathrm{dist}(v,u)$.

The definition of strong spatial mixing we use is from~\cite{MR2277139} and~\cite{weitzthesis} (we state the
definition only for colorings).

\begin{definition}\label{def:ssm}
Let $\delta:\N \to \R^+$. The infinite-volume Gibbs measure $p$ on $q$-colorings of $G=(V,E)$ has strong spatial mixing with rate $\delta(\cdot)$ if and only if for every vertex $v$, every $U \subseteq V$, and every pair of $q$-colorings $\sigma_U,\phi_U$,
$$
|p_v^{\sigma_U}-p_v^{\phi_U}| \leq \delta(\mathrm{dist}(v,\Delta)),
$$
where $\Delta \subseteq U$ is the set of vertices on which $\sigma_U$ and $\phi_U$ differ.
\end{definition}

Recently, strong spatial mixing received attention because of its connection with efficient approximation algorithms for certain spin systems (c.f.~\cite{MR2277139,NT07}). For colorings of graphs, strong spatial mixing results were established for different lattice graphs
~\cite{MR2191453, MR2570925}.

A Cayley tree (also known as Bethe lattice) $\widehat{T}^b$ is an infinite $(b+1)$-regular tree.
In this paper, we prove the strong spatial mixing for $q$-colorings on Cayley trees $\widehat{T}^b$.

\begin{theorem}\label{thm:main}
For $q \geq 1+\lceil cb \rceil$ where $c \approx 1.764$ is the root of $c=\exp(1/c)$, the infinite-volume Gibbs measure $p$ on $q$-colorings of $\widehat{T}^b$ has strong spatial mixing with rate $\delta(d)=C \exp(-a d)$ for some positive constants $C$ and $a$.
\end{theorem}

We also establish the strong spatial mixing of $q$-colorings on binary trees, for $q=4$.

\begin{theorem}\label{thm:main4}
Let $q = 4$. The infinite-volume Gibbs measure $p$ on $q$-colorings of $\widehat{T}^2$ has strong spatial mixing with rate $\delta(d)=C \exp(-a d)$ for some positive constants $C$ and $a$.
\end{theorem}

We will prove Theorem~\ref{thm:main} and Theorem~\ref{thm:main4} by analyzing the sum-product algorithm, which we review in the next section.

\section{The sum-product algorithm}

Let $T=(V,E)$ be a $b$-ary tree, $U \subseteq V$ be a subset of vertices, and $\sigma_U: U \to [q]$ be a $q$-coloring on
the vertices in $U$. For every vertex $v \in V$, a message (according to $\sigma_U$) from $v$ to its parent is a probability
distribution $\alpha \in \R^q$ on $[q]$ where $\alpha_i$ is proportional to the number of $q$-colorings of the subtree
rooted at $v$ such that the color of $v$ is different from $i$. The message from $v$ to its parent can also be defined
recursively as follows.
\begin{itemize}
\item If $v \in U$ and $\sigma_U(v)=k$ for some $k \in [q]$, then for $i \in [q]$,
$$\alpha_i = \left\{
\begin{array}{ll}
0, &\mbox{if } i=k,\\
1/(q-1), &\mbox{otherwise.}
\end{array}
\right.$$ \item If $v \in V \setminus U$ and $v$ is a leaf, then for $i \in [q]$, $\alpha_i = 1/q$. \item If $v \in V \setminus
U$ and $v$ is not a leaf, let $\beta^\ell$, $\ell \in [b]$, be the message from the $\ell$-th child of $v$ to $v$. Then $\alpha =
f(\beta^1,\ldots,\beta^b)$, where $f: (\R^{q})^b \to \R^q$ is defined by
\begin{equation}\label{eq:mproc}
\left(f(\beta^1,\ldots,\beta^b)\right)_i = \frac{\sum_{j \in [q], j \neq i} \prod_{\ell=1}^b \beta^\ell_j}{(q-1)\sum_{j \in [q]} \prod_{\ell=1}^b \beta^\ell_j}.
\end{equation}
\end{itemize}
Note that the right-hand side of (\ref{eq:mproc}) is always bounded by $1/(q-1)$ and
hence all messages are in the set $\S_1$, where $\S_1$ is the set of vectors $\gamma \in \R^q$ satisfying
\begin{equation}\label{eq:ppty1}
\sum_{i=1}^q \gamma_i=1, \quad\mbox{and}\quad 0 \leq \gamma_i \leq \frac{1}{q-1}, \mbox{ for all }i \in [q].
\end{equation}

The following folklore result gives a connection between strong spatial mixing and
the sum-product algorithm.

\begin{lemma}\label{lem:defequiv}
Assume that there exists a function $\delta$ such that for every $b$-ary tree $T=(V,E)$ (with root $r$), for any subset of
vertices $U\subseteq V$, and any pair of configurations $\sigma_U,\phi_U: U \to [q]$, the message $\alpha$ from $u$ (a child of
$r$) to $r$ according to $\sigma_U$ and the message $\beta$ from $u$  to $r$ according to $\phi_U$ satisfy
\begin{equation}\label{erat}
\| \alpha-\beta \| \leq \delta({\rm dist}(r,\Delta)),
\end{equation}
where $\|\cdot\|$ is some fixed norm and $\Delta\subseteq U$ is a set where $\sigma_U$ and
$\phi_U$ differ. Then the infinite-volume Gibbs measure $p$ on $q$-colorings of $\widehat{T}^b$
has strong spatial mixing with rate $C\delta(d)$ for some positive constant $C$ (the constant $C$ depends on $b$ and
the norm used).
\end{lemma}

We need the following property of the messages in $\S_1$.

\begin{lemma}\label{lem:prod}
Let $\alpha^1,\ldots,\alpha^b \in \S_1$, then 
$$
\sum_{j \in [q]}\prod_{i \in [b]} \alpha^i_j \geq \frac{q-b}{(q-1)^b}.
$$
\end{lemma}

\begin{proof}
We prove the statement by induction on $b$. For $b=1$, the statement is true.

We assume that the statement is true for $b=t \geq 1$. We now prove the statement for $b=t+1$. Let $z_j = \prod_{i=1}^t \alpha^i_j$, for $j \in [q]$. We assume, w.l.o.g., that $z_1 \leq \ldots \leq z_q$. We have
\begin{equation}\label{eq:lemdir1}
\sum_{j \in [q]}\prod_{i \in [b]} \alpha^i_j = \sum_{j \in [q]} z_j \alpha^{t+1}_j \geq z_q \cdot 0 + \sum_{j=1}^{q-1} z_j/(q-1).
\end{equation}
Fixing $\alpha^{t+1}=(1/(q-1),\ldots,1/(q-1),0)$, we have $\sum_{j \in [q]}\prod_{i \in [b]} \alpha^i_j$ is minimized when $z_q$ is maximized. Note that $z_q = \prod_{i=1}^t \alpha^i_q \leq 1/(q-1)^b$. Hence we have $\sum_{j \in [q]}\prod_{i \in [b]} \alpha^i_j$ is minimized when $\alpha^i_q=1/(q-1)$, for all $i \in [t]$.

We next bound $\sum_{j=1}^{q-1} z_j$ from below. Let $\beta^i_j = \alpha^i_j(q-1)/(q-2)$, for $i \in [t]$ and $j \in [q-1]$. Note that $\sum_{j \in [q-1]} \beta^i_j=1$ and $0 \leq \beta^i_j \leq 1/(q-2)$, for $i \in [t]$ and $j \in [q-1]$. By induction hypothesis, we have 
$$
\sum_{j \in [q-1]}\prod_{i \in [t]} \beta^i_j \geq \frac{q-1-t}{(q-2)^{t}}.
$$
Hence we have
\begin{equation}\label{eq:lemdir2}
\sum_{j=1}^{q-1} z_j = \sum_{j \in [q-1]} \prod_{i \in [t]} \alpha^i_j = \sum_{j \in [q-1]} \prod_{i \in [t]} \frac{\beta^i_j(q-2)}{q-1} = \frac{(q-2)^t}{(q-1)^t} \sum_{j \in [q-1]} \prod_{i \in [t]} \beta^i_j \geq \frac{q-1-t}{(q-1)^t}.
\end{equation}
By~(\ref{eq:lemdir1}) and~(\ref{eq:lemdir2}), we have
$$
\sum_{j \in [q]}\prod_{i \in [b]} \alpha^i_j \geq \frac{q-b}{(q-1)^b}.
$$
\end{proof}

\section{The messages in the sum-product algorithm contract}

\subsection{Case $q \geq 1+\lceil cb \rceil$}

Theorem~\ref{thm:main} will follow from Lemma~\ref{lem:defequiv} and
the following lemma, which shows that~(\ref{eq:mproc}) is
a contraction in the following sense: if in a node we have a pair of messages from
each child then the pair of
messages from the node (where the $i$-th component in the pair is obtained
by applying~(\ref{eq:mproc}) to the $i$-th components of pairs from the children)
is closer in the $\ell_1$-norm than the $\ell_1$-distance within the pair from at least one child.

\begin{lemma}\label{lem:contractb}
Let $T=(V,E)$ be a $b$-ary tree rooted at $r$. Let $w\neq r$ be a vertex of $T$ and
let $u^1,\ldots,u^b$ be the $b$ children of $w$. Let $U \subseteq V$. Let
$\sigma_U,\phi_U: U \to [q]$ be a pair of configurations
such that $\mathrm{dist}(w,\Delta) \geq 1$, where $\Delta \subseteq U$ is
the set of vertices on which $\sigma_U$ and $\phi_U$ differ. For $\ell \in [b]$, let $\alpha^\ell$ and $\beta^\ell$ be the messages from
$u^\ell$ to $w$ according to $\sigma_U$ and $\phi_U$, respectively. Then the messages
$\zeta$ and $\eta$ from $w$ according to $\sigma_U$ and $\phi_U$, respectively, satisfy
$$
\|\zeta-\eta\|_{1} \leq \frac{b}{q}\left(1-\frac{1}{q-b}\right)^{-b+b^2/q} \cdot \max_{\ell \in [b]} \|\alpha^\ell-\beta^\ell\|_{1}.
$$
\end{lemma}

\begin{remark}
In the previous version of the paper, we stated an incorrect version of 
Lemma~\ref{lem:contractb} using $\ell_\infty$-norm, thanks to Sidhant Misra
and David Gamarnik for pointing out the error.
\end{remark}

\begin{proof}[Proof of Theorem~\ref{thm:main}]
We will claim that $\frac{b}{q}\left(1-\frac{1}{q-b}\right)^{-b+b^2/q} < 1$ when $q \geq 1+\lceil cb \rceil$, where $c>0$ is the root of $\exp(1/c)=c$. Taking the derivative of $\frac{b}{q}\left(1-\frac{1}{q-b}\right)^{-b+b^2/q}$ w.r.t.\ $q$, we obtain
\begin{equation}\label{eq:dircont}
\frac{b\left(-q^2+q+b^2(q-b-1)\ln\left(\frac{q-b}{q-b-1}\right)\right)}{q^3(q-b-1)} \left(\frac{q-b-1}{q-b}\right)^{b(b-q)/q}.
\end{equation}
We will show that~(\ref{eq:dircont}) is not positive when $q \geq b+1$. It is sufficient to prove that $-q^2+q+b^2(q-b-1)\ln\left(\frac{q-b}{q-b-1}\right) \leq 0$. Note that $\ln(1+x) \leq x$ for all $x \geq 0$. Hence we have
$$
b^2(q-b-1)\ln\left(\frac{q-b}{q-b-1}\right) \leq b^2 \leq (q-1)^2 \leq q^2-q.
$$
We now prove that $\frac{b}{q}\left(1-\frac{1}{q-b}\right)^{-b+b^2/q} < 1$ when $q = 1+ cb$ and $b\geq 2$. Let
$$
g(b)=\frac{b}{cb+1}\left(1-\frac{1}{(c-1)b+1}\right)^{-b+b^2/(cb+1)}.
$$
Taking the derivative of $g$, we have
\begin{equation*}
\begin{split}
\frac{\mathrm{d}g}{\mathrm{d}b} =& -\frac{1}{(cb+1)^3}\left(\frac{(c-1)b}{(c-1)b+1}\right)^{-\frac{b((c-1)b+1)}{cb+1}}\\
&\cdot\left(b+cb^2-1-cb+(b+2cb^2+c^2b^3-2b^2-cb^3)\ln\left(\frac{(c-1)b}{(c-1)b+1}\right)\right).
\end{split}
\end{equation*}
We will show that $\frac{\mathrm{d}g}{\mathrm{d}b} > 0$ for $b \geq 2$. It is sufficient to prove that
$$
b+cb^2-1-cb+(b+2cb^2+c^2b^3-2b^2-cb^3)\ln\left(\frac{(c-1)b}{(c-1)b+1}\right) < 0.
$$
Note that $\ln(1+x) \geq x-x^2/2$ for all $x \geq 0$. We have
\begin{eqnarray*}
&&-b-cb^2+1+cb+(b+2cb^2+c^2b^3-2b^2-cb^3)\ln\left(1+\frac{1}{(c-1)b}\right)\\
&\geq&-b-cb^2+1+cb+(b+2cb^2+c^2b^3-2b^2-cb^3)\left(\frac{1}{(c-1)b}-\frac{1}{2(c-1)^2b^2}\right)\\
&=&\frac{(2c^3-3c^2-c+2)b^2+(2c^2-4c+2)b-1}{2(c-1)^2b}\\
&>& 0,
\end{eqnarray*}
for all $b \geq 2$. Hence, $\frac{\mathrm{d}g}{\mathrm{d}b} > 0$ for $b \geq 2$. Note that $g(b) \to 1$ as $b \to \infty$. We have $g(b)<1$ for $b \geq 2$.

Hence Theorem~\ref{thm:main} follows from Lemma~\ref{lem:defequiv} and Lemma~\ref{lem:contractb}.
\end{proof}

Before proving Lemma~\ref{lem:contractb}, we need a more detailed understanding of the messages. Let $\S'_1 \subseteq \S_1$
be the set of vectors $\gamma \in \R^q$ satisfying the following property:
\begin{equation}\label{mp1}
\mbox{for every $i \in [q]$ we have $\frac{1}{q-1}\left(1-\frac{1}{q-b}\right) \leq \gamma_i \leq \frac{1}{q-1}$.}
\end{equation}
Let $\S_2$ be the set of permutations of $(0,1/(q-1),\ldots,1/(q-1))$.

\begin{claim}\label{clm:sp}
Let $\gamma \in \S'_1$, then $\gamma$ has at most $b$ entries of value $1/(q-1)$. If $\gamma$ has $b$ entries of value $1/(q-1)$, then all other entries of $\gamma$ have value $(1-1/(q-b))/(q-1)$.
\end{claim}
\begin{proof}
Assume that $\gamma$ has $s$ entries of $1/(q-1)$ and $\gamma_1,\ldots,\gamma_s = 1/(q-1)$. Then by~(\ref{mp1}), we have
$$
1=\sum_{j \in [q]} \gamma_j = \frac{s}{q-1} + \sum_{j=s+1}^q \gamma_j \geq \frac{s}{q-1} + \frac{q-s}{q-1}\left(1-\frac{1}{q-b}\right) = \frac{1}{q-1}\left(q-\frac{q-s}{q-b}\right),
$$
which implies $s \leq b$. Note that if $s=b$, we have $\gamma_{b+1}=\ldots=\gamma_q = (1-1/(q-b))/(q-1)$.
\end{proof}

The following lemma shows that the set $\S_1' \cup \S_2$ contains all the possible messages.

\begin{lemma}\label{lem:closure}
For every $\beta^1,\ldots,\beta^b \in \S_1' \cup \S_2$, we have $f(\beta^1,\ldots,\beta^b) \in \S_1'$.
\end{lemma}

\begin{proof}
To establish~(\ref{mp1}) we use Lemma~\ref{lem:prod} and the fact $0 \leq \beta^1_i,\ldots,\beta^b_i \leq 1/(q-1)$:
\begin{equation*}
\left(f(\beta^1,\ldots,\beta^b)\right)_i = \frac{1}{q-1}\left(1-\frac{\prod_{\ell=1}^b \beta^\ell_i}{\sum_{j \in [q]} \prod_{\ell=1}^b \beta^\ell_j}\right) \geq \frac{1}{q-1}\left(1-\frac{1}{q-b}\right).
\end{equation*}
\end{proof}

Lemma~\ref{lem:contractb} follows from the following two lemmas:

\begin{lemma}\label{lem:ineqb}
For every $\gamma^1,\ldots,\gamma^{b-1} \in \S'_1 \cup \S_2$ and every $\alpha,\beta \in \S'_1$, we have
$$
\|f(\gamma^1,\ldots,\gamma^{b-1},\alpha)-f(\gamma^1,\ldots,\gamma^{b-1},\beta)\|_{1}\leq \frac{1}{(q-1)^b A} \|\alpha-\beta\|_{1},
$$
where $A=\sum_{j \in [q]} z_j \alpha_j$ and $z_j = \prod_{i \in [b-1]} \gamma^i_j$, for $j \in [q]$.
\end{lemma}

\begin{lemma}\label{lem:bb}
For every $0 \leq s \leq b < q$, let $\alpha^1,\ldots,\alpha^{b-s} \in \S'_1$ and $\alpha^{b-s+1},\ldots,\alpha^b \in \S_2$, we have 
\begin{equation}\label{eq:innerb}
\sum_{j \in [q]}\prod_{i \in [b]} \alpha^{i}_j \geq \frac{q-s}{(q-1)^{b}}\left(1-\frac{1}{q-b}\right)^{b-s-(b-s)^2/(q-s)}.
\end{equation}
\end{lemma}

We first prove Lemma~\ref{lem:contractb}, and then Lemma~\ref{lem:ineqb} and Lemma~\ref{lem:bb}.

\begin{proof}[Proof of Lemma~\ref{lem:contractb}]
If $w \in U$ then from the assumption $\mathrm{dist}(w,\Delta) \geq 3$ we have
$\sigma_U(w)=\phi_U(w)$ and hence $\zeta = \eta$.
From now on we assume that $w \not\in U$ and thus $\zeta=f(\alpha^1,\ldots,\alpha^b)$
and $\eta=f(\beta^1,\ldots,\beta^b)$.

Let $s$ be the number of children of $w$ which are in $U$. W.l.o.g., we assume that $u^1,\ldots,u^s$ are in $U$. Note that $\mathrm{dist}(u^\ell,\Delta) \geq 2$ for all $\ell \in [b]$. We have $\sigma_U(u^\ell)=\phi_U(u^\ell)$ for every $\ell \in [s]$, which implies $\alpha^\ell = \beta^\ell \in \S_2$, for every $\ell \in [s]$, and $\alpha^\ell,\beta^\ell \in \S'_1$, for every $s< \ell \leq b$. Using triangle inequality, Lemma~\ref{lem:ineqb} and Lemma~\ref{lem:bb} we obtain
\begin{eqnarray*}
&& \|f(\alpha^1,\ldots,\alpha^b)-f(\beta^1,\ldots,\beta^b)\|_{1}\\
& \leq &\sum_{\ell=s}^{b-1} \|f(\beta^1,\ldots,\beta^\ell,\alpha^{\ell+1},\ldots,\alpha^b)-f(\beta^1,\ldots,\beta^{\ell+1},\alpha^{\ell+2},\ldots,\alpha^b)\|_{1}\\
& \leq &\frac{1}{q-s} \left(1-\frac{1}{q-b}\right)^{(b-s)^2/(q-s)-(b-s)} \sum_{\ell=s+1}^b \|\alpha^\ell-\beta^\ell\|_{1}\\
& \leq &\frac{b-s}{q-s}\left(1-\frac{1}{q-b}\right)^{(b-s)^2/(q-s)-(b-s)}  \max_{\ell \in [b]} \|\alpha^\ell-\beta^\ell\|_{1}\\
& \leq &\frac{b}{q}\left(1-\frac{1}{q-b}\right)^{-b+b^2/q} \max_{\ell \in [b]} \|\alpha^\ell-\beta^\ell\|_{1},
\end{eqnarray*}
where the last inequality follows from the facts that $(b-s)/(q-s)$ as a function of $s$ is monotonically decreasing for $0 \leq s \leq b < q$ and $(b-s)^2/(q-s)-(b-s)$ as a function of $s$ is monotonically increasing for $0 \leq s \leq b < q$.
\end{proof}

We now prove Lemma~\ref{lem:ineqb}.

\begin{proof}[Proof of Lemma~\ref{lem:ineqb}]
We will show that for every $s_k \in\{\pm 1\}$, $k \in [q]$, we have
$$
\sum_{k \in [q]} s_k \cdot \left( (f(\gamma^1,\ldots,\gamma^{b-1},\alpha))_k-(f(\gamma^1,\ldots,\gamma^{b-1},\beta))_k \right) \leq \frac{1}{(q-1)^b A} \|\alpha-\beta\|_{1}.
$$
Let
\begin{equation}\label{eq:bee0}
\begin{split}
Q(t,\alpha,\beta) :=& \| \alpha - ((1-t)\alpha + t \beta)\|_{1}\\
&- (q-1)^bA \sum_{k \in [q]} s_k \cdot \left((f(\gamma^1,\ldots,\gamma^{b-1},\alpha))_k-(f(\gamma^1,\ldots,\gamma^{b-1}, (1-t)\alpha + t\beta))_k \right).
\end{split}
\end{equation}
Note that we have $Q(0,\alpha,\beta)=0$ and our goal is to lower bound $Q(1,\alpha,\beta)$. We have
\begin{eqnarray*}
P(t,\alpha,\beta) &:=& \frac{\partial}{\partial t} Q(t,\alpha,\beta)\\
&=& \|\alpha-\beta\|_{1} - \frac{(q-1)^{b-1}A}{B^2}
\sum_{k \in [q]} s_kz_k \left(\alpha_k\left(\sum_{j=1}^q z_j \beta_j\right) - \beta_k\left(\sum_{j=1}^q z_j \alpha_j\right)\right),
\end{eqnarray*}
where
$$
B=\sum_{j=1}^q z_j ((1-t)\alpha_j + t \beta_j).
$$
We are going to lower bound $P(t,\alpha,\beta)$ for all $\alpha,\beta\in \S'_1$ and $t\in [0,1)$. We have
$$
P(t,\alpha,\beta)=\frac{1}{1-t} P(0,(1-t)\alpha+t\beta,\beta),
$$
and hence it is enough to consider the case $t=0$ (note that $\S'_1$ is convex, and hence $(1-t)\alpha+t\beta$ is in $\S'_1$
if $\alpha,\beta$ are in $\S'_1$). Substituting $\beta_j = \alpha_j + \varepsilon_j$ into $P(0,\alpha,\beta)$ we obtain
\begin{eqnarray*}
P(0,\alpha,\beta) &=& \sum_{j \in [q]} |\varepsilon_j| - \frac{(q-1)^{b-1}}{A}
\sum_{k \in [q]} s_k z_k\left( \left(\sum_{j \in [q], j \neq k} \alpha_k z_j \varepsilon_j \right) - \varepsilon_k \left(\sum_{j \in [q], j \neq k} z_j \alpha_j \right) \right)\\
&=& \sum_{j \in [q]} |\varepsilon_j| - \frac{(q-1)^{b-1}}{A}
\sum_{k \in [q]} \varepsilon_k z_k \sum_{j \in [q]} (s_j-s_k)z_j \alpha_j\\
&=& \sum_{j \in [q]} |\varepsilon_j| - \frac{(q-1)^{b-1}}{A}
\sum_{k \in [q]} \varepsilon_k \tau_k,
\end{eqnarray*}
where $\tau_k := z_k \sum_{j \in [q]} (s_j-s_k)z_j \alpha_j$.

Let $\varepsilon_i^+=\max\{\varepsilon_i,0\}$, $\varepsilon_i^-=\max\{-\varepsilon_i,0\}$, and
$D=\sum_{i=1}^q \varepsilon^+_i=\sum_{i=1}^q \varepsilon^-_i=\|\alpha-\beta\|_1/2$.
We have
\begin{eqnarray*}
P(0,\alpha,\beta) &=& 2D - \frac{(q-1)^{b-1}}{A} \sum_{k \in [q]} (\varepsilon_k^+ - \varepsilon_k^-)\tau_k\\
&\geq& 2D - \frac{(q-1)^{b-1}}{A} \left( \sum_{k \in [q]}\varepsilon_k^+ \max_{k \in [q]}\tau_k -\sum_{k \in [q]} \varepsilon_k^- \min_{k \in [q]}\tau_k \right)\\
&=& 2D - \frac{(q-1)^{b-1}D}{A} \left(\max_{k \in [q]}\tau_k -\min_{k \in [q]}\tau_k \right).
\end{eqnarray*}

\begin{claim}\label{clm:tau}
$$\max_{k \in [q]}\tau_k -\min_{k \in [q]}\tau_k \leq \frac{2A}{(q-1)^{b-1}}.$$
\end{claim}

\begin{proof}[Proof of Claim~\ref{clm:tau}]
We assume, w.l.o.g., that the largest $\tau_k$ is $\tau_q$ and the smallest $\tau_k$ is
$\tau_1$. We will show
\begin{equation}\label{e2}
\tau_q-\tau_1 = \sum_{j=1}^q
(z_q(s_j-s_q)-z_1(s_j-s_1)) z_j \alpha_j \leq \frac{2A}{(q-1)^{b-1}}.
\end{equation}
Note that $s_q$ occurs in~(\ref{e2}) with
negative sign and hence we can assume $s_q=-1$. Similarly $s_1$
occurs in~(\ref{e2}) with positive sign and hence
we can assume $s_1=+1$.

Let $P$ be the set of $j\in [q]$ such that $s_j=+1$. Let $\overline{P}=[q]\setminus P$. We have
$\{1\}\subseteq P$ and $\{q\} \subseteq \overline{P}$. We can rewrite~(\ref{e2}) as follows
\begin{equation}\label{ee2}
\tau_q-\tau_1 = 2z_q \sum_{j\in P} z_j \alpha_j +
2z_1\sum_{j\in\overline{P}} z_j \alpha_j.
\end{equation}

Note that the right-hand side of~(\ref{ee2}) is symmetric between $z_1$ and  $z_q$
and hence we can, w.l.o.g., assume $z_q \geq z_1$. For fixed $\alpha,z$ the right
hand-side of~(\ref{ee2}) is maximized when $P=[q-1]$. Hence we have
$$
\tau_q-\tau_1 \leq 2z_q ( A - z_q\alpha_q) + 2z_1 z_q\alpha_q =
2z_qA - 2(z_q-z_1)z_q\alpha_q \leq \frac{2A}{(q-1)^{b-1}},
$$
where in the last inequality we used $z_1 \leq z_q \leq 1/(q-1)^{b-1}$.
\end{proof}

We now continue the proof of Lemma~\ref{lem:ineqb}. By Claim~\ref{clm:tau}, we have
$$
P(0,\alpha,\beta) \geq 0.
$$
Hence we have
\begin{equation}\label{eq:bee2}
Q(1,\alpha,\beta) = \int_0^1 P(t,\alpha,\beta) \mathrm{d}t = \int_0^1 \frac{1}{1-t} P(0,(1-t)\alpha+t\beta,\beta) \mathrm{d}t \geq 0.
\end{equation}
From~(\ref{eq:bee0}) and~(\ref{eq:bee2}), we obtain
$\|\alpha-\beta\|_{1} \geq (q-1)^bA \|f(\gamma^1,\ldots,\gamma^{b-1},\alpha)-f(\gamma^1,\ldots,\gamma^{b-1},\beta) \|_{1}$.
\end{proof}

Before proving Lemma~\ref{lem:bb}, we will show that the inequality of Lemma~\ref{lem:prod}
can be strengthened if we assume that $\alpha^1,\ldots,\alpha^b \in \S'_1$.

\begin{lemma}\label{lem:prodnew}
Let $\alpha^1,\ldots,\alpha^b \in \S'_1$, we have 
\begin{equation}\label{eq:inner}
\sum_{j \in [q]}\prod_{i \in [b]} \alpha^i_j \geq \frac{q}{(q-1)^b}\left(1-\frac{1}{q-b}\right)^{b-b^2/q}.
\end{equation}
\end{lemma}

\begin{proof}
We first claim that the LHS of~(\ref{eq:inner}) is minimized when $\alpha^\ell$ has $b$ entries of value $1/(q-1)$ for all $\ell \in [b]$. Fix $\alpha^1,\ldots,\alpha^{b-1}$. Let $z_j=\prod_{i \in [b-1]} \alpha^i_j$, for $j \in [q]$. W.l.o.g., we assume that $z_1 \leq z_2 \leq \ldots \leq z_{q}$. Then by Claim~\ref{clm:sp}, the LHS of~(\ref{eq:inner}) is minimized when $\alpha^b_1=\ldots=\alpha^b_b=1/(q-1)$ and $\alpha^b_{b+1}=\ldots=\alpha^b_q=(1-1/(q-b))/(q-1)$. By the same claim, the LHS of~(\ref{eq:inner}) is minimized when $\alpha^\ell$ has $b$ entries of value $1/(q-1)$ for all $\ell \in [b]$. Hence we can assume that $\alpha^\ell$ has $b$ entries of value $1/(q-1)$ for all $\ell \in [b]$.

We next claim that the LHS of~(\ref{eq:inner}) is minimized when the number of $1/(q-1)$ in $(\alpha^\ell_j)_{\ell \in [b]}$ is either $\lfloor b^2/q \rfloor$ or $\lfloor b^2/q \rfloor+1$, for every $j \in [q]$.
Let $j_1,j_2 \in [q]$ and $j_1 \neq j_2$. Fix $(\alpha^\ell_j)_{\ell \in [b]}$ for all $j \in [q]\setminus \{j_1,j_2\}$. Let $t_1$ be the number of $1/(q-1)$ in $(\alpha^\ell_{j_1})_{\ell \in [b]}$ and let $t_2$ be the number of $1/(q-1)$ in $(\alpha^\ell_{j_2})_{\ell \in [b]}$. W.l.o.g., we assume that $t_1 \leq t_2$. We claim that the LHS of~(\ref{eq:inner}) is minimized when $t_2-t_1 \leq 1$. We have
\begin{equation*}
\begin{split}
\prod_{\ell \in [b]}\alpha^\ell_{j_1} + \prod_{\ell \in [b]}\alpha^\ell_{j_2} = 
\frac{1}{(q-1)^b}\left(1-\frac{1}{q-b}\right)^{b-t_1}+\frac{1}{(q-1)^b}\left(1-\frac{1}{q-b}\right)^{b-t_2}\\
= \frac{1}{(q-1)^b}\left(1-\frac{1}{q-b}\right)^{b-t_1-1}+\frac{1}{(q-1)^b}\left(1-\frac{1}{q-b}\right)^{b-t_2+1}\\
 + \frac{1}{(q-1)^b(q-b)}\left(1-\frac{1}{q-b}\right)^{b-t_1-1}\left(\left(1-\frac{1}{q-b}\right)^{t_1+1-t_2}-1\right).
\end{split}
\end{equation*}
If $t_2-t_1 > 1$, then
$$
\prod_{\ell \in [b]}\alpha^\ell_{j_1} + \prod_{\ell \in [b]}\alpha^\ell_{j_2} > \frac{1}{(q-1)^b}\left(1-\frac{1}{q-b}\right)^{b-t_1-1}+\frac{1}{(q-1)^q}\left(1-\frac{1}{q-b}\right)^{b-t_2+1},
$$
and the LHS of~(\ref{eq:inner}) becomes smaller by moving one $1/(q-1)$ from $(\alpha^\ell_{j_2})_{\ell \in [b]}$ to $(\alpha^\ell_{j_1})_{\ell \in [b]}$.

The minimum value of the LHS of~(\ref{eq:inner}) is:
\begin{equation}\label{eq:cont}
\frac{1}{(q-1)^b}\left(1-\frac{1}{q-b}\right)^{b-\lfloor b^2/q \rfloor} (q-b^2+q\lfloor b^2/q \rfloor) + \frac{1}{(q-1)^b}\left(1-\frac{1}{q-b}\right)^{b-\lfloor b^2/q \rfloor-1}(b^2-q\lfloor b^2/q \rfloor).
\end{equation}
Let $\{b^2/q\}$ be the fractional part of $b^2/q$, we can rewrite~(\ref{eq:cont}) as
\begin{equation*}
\begin{split}
&\frac{1}{(q-1)^b}\left(1-\frac{1}{q-b}\right)^{b-b^2/q+\{b^2/q\}} q(1-\{b^2/q\})+\frac{1}{(q-1)^b}\left(1-\frac{1}{q-b}\right)^{b-b^2/q+\{b^2/1\}-1}q\{b^2/q\}\\
&=\frac{q}{(q-1)^b}\left(1-\frac{1}{q-b}\right)^{b-b^2/q}\left(\left(1-\frac{1}{q-b}\right)^{\{b^2/q\}}(1-\{b^2/q\})+\left(1-\frac{1}{q-b}\right)^{\{b^2/q\}-1}\{b^2/q\}\right)\\
&\geq \frac{q}{(q-1)^b}\left(1-\frac{1}{q-b}\right)^{b-b^2/q},
\end{split}
\end{equation*}
where the last inequality follows from the fact that $x^y(1-y)+x^{y-1}y \geq 1$ for $0 < x < 1$ and $0 \leq y < 1$.
\end{proof}

We now prove Lemma~\ref{lem:bb}.

\begin{proof}[Proof of Lemma~\ref{lem:bb}]
We prove the statement by induction on $s$. For $s=0$, the statement follows from Lemma~\ref{lem:prodnew}.

We assume that the statement is true for $s=t \geq 0$. We next consider the case when $s=t+1$. W.l.o.g., we assume $\alpha^{b}=(1/(q-1),\ldots,1/(q-1),0)$. The LHS of~(\ref{eq:innerb}) is minimized when $\alpha^\ell_q = 1/(q-1)$ for $\ell \in [b-1]$.

We define $\beta^\ell_j = \alpha^\ell_j (q-1)/(q-2)$, for $j \in [q-1]$ and $\ell \in [b-1]$. Note that $\sum_{j \in [q]} \beta^\ell_j = 1$, for $\ell \in [b-1]$, and $0 \leq \beta^\ell_j \leq 1/(q-2)$, for $b-s+1 \leq \ell \leq b-1$, and $(1-1/(q-b))/(q-2) \leq \beta^\ell_j \leq 1/(q-2)$, for $\ell \in [b-s]$. By induction hypothesis, we have
$$
\sum_{j \in [q-1]} \prod_{\ell \in [b-1]} \beta^\ell_j \geq \frac{q-1-t}{(q-2)^{b-1}}\left(1-\frac{1}{q-b}\right)^{b-1-t-(b-1-t)^2/(q-1-t)}.
$$
Hence the LHS of~(\ref{eq:innerb}) is lower bounded by:
\begin{equation*}
\begin{split}
\frac{1}{q-1}\sum_{j \in [q-1]} \prod_{\ell \in [b-1]} \alpha^\ell_j = \frac{1}{q-1}\sum_{j \in [q-1]} \prod_{\ell \in [b-1]} \frac{(q-2)\beta^\ell_j}{q-1} = \frac{(q-2)^{b-1}}{(q-1)^b} \sum_{j \in [q-1]} \prod_{\ell \in [b-1]} \beta^\ell_j\\
\geq \frac{q-s}{(q-1)^b}\left(1-\frac{1}{q-b}\right)^{b-s-(b-s)^2/(q-s)}.
\end{split}
\end{equation*}
\end{proof}

\subsection{Case $q=4$ and $b=2$}

In this section, we assume that $q=4$ and $b=2$. We will prove the following strengthening of Lemma~\ref{lem:contractb} for the special case $q=4$ and $b=2$.

\begin{lemma}\label{lem:contract1}
Let $T$ be a binary tree rooted at $r$. Let $w\neq r$ be a vertex of $T$ and
let $u$ and $u'$ be the two children of $w$. Let $U \subseteq V$ and let
$\sigma_U,\phi_U: U \to [4]$ be a pair of configurations such that
$\mathrm{dist}(w,\Delta) \geq 3$, where $\Delta \subseteq U$
is the set of vertices on which $\sigma_U$ and $\phi_U$ differ.
Let $\alpha$, $\beta$ be the messages
from $u$ to $w$ according to $\sigma_U$ and $\phi_U$, respectively, and let $\alpha'$ and $\beta'$ be
the messages from $u'$ to $w$ according to $\sigma_U$ and $\phi_U$, respectively. Then the messages
$\zeta$ and $\eta$ from $w$ according to $\sigma_U$ and $\phi_U$, respectively, satisfy
\begin{equation}\label{ehe}
\|\zeta-\eta\|_1 \leq \frac{48}{49} \cdot \max\{\|\alpha-\beta\|_1,\|\alpha'-\beta'\|_1\}.
\end{equation}
\end{lemma}

Theorem~\ref{thm:main4} now follows:

\begin{proof}[Proof of Theorem~\ref{thm:main4}]
Theorem~\ref{thm:main4} follows from Lemma~\ref{lem:defequiv} and Lemma~\ref{lem:contract1}.
\end{proof}

Before proving Lemma~\ref{lem:contract1}, we need a more detailed understanding of the messages.
Let $\S_1'\subseteq \S_1$ be the set of vectors $\gamma \in \R^4$ satisfying the following three
properties:
\begin{equation}\label{4mp1}
\mbox{for every $i \in [4]$ we have $1/6 \leq \gamma_i \leq 1/3$,}
\end{equation}
\begin{equation}\label{4mp2}
\mbox{for every $i \in [4]$ either $\gamma_i = 1/3$ or $\gamma_i \leq 11/36$,}
\end{equation}
\begin{equation}\label{4mp3}
\mbox{if $\gamma$ has exactly two entries of value $1/3$, then $\gamma$ is a permutation of $(1/6,1/6,1/3,1/3)$.}
\end{equation}

Let $\S_2$ be the set of permutations of $(0,1/3,1/3,1/3)$.
\begin{definition}
We say that two vectors $\gamma,\xi\in\S'_1\cup\S_2$ are {\em coupled} if for every $i\in [4]$ we have
$\gamma_i=1/3$ if and only if $\xi_i=1/3$.
\end{definition}

\begin{claim}\label{clm:prodbd}
Let $\gamma,\xi \in \S_1$. Then
$\sum_{i=1}^4 \gamma_i \xi_i \leq \frac{1}{3}$.
\end{claim}

\begin{proof}
W.l.o.g., we assume that $\gamma_1 \leq \ldots \leq \gamma_4$. For fixed $\gamma$, the maximum of $\sum_{i=1}^4 \gamma_i \xi_i$
over $\xi\in\S_1$ happens for $\xi=(0,1/3,1/3,1/3)$ and hence
$$
\sum_{i=1}^4 \gamma_i \xi_i \leq \gamma_1 \cdot 0 + \sum_{i=2}^4 \gamma_i/3 = (1-\gamma_1)/3 \leq \frac{1}{3}.
$$
\end{proof}

The following lemma shows that the set $\S_1' \cup \S_2$ contains all the possible messages.
\begin{lemma}\label{lem:closure4}
For every $\gamma,\xi \in \S_1' \cup \S_2$, we have $f(\gamma,\xi) \in \S_1'$.
\end{lemma}

\begin{proof}
To establish~(\ref{4mp1}) we use Lemma~\ref{lem:prod} and the fact $0 \leq \gamma_i,\xi_i \leq 1/3$:
$$
(f(\gamma,\xi))_i = \frac{1}{3} \left(1-\gamma_i \xi_i / \left(\sum_{j=1}^4 \gamma_j \xi_j\right) \right) \geq \frac{1}{3}\left(1-\frac{1}{2}\right) = 1/6.
$$
Note that if $(f(\gamma,\xi))_i \neq 1/3$ then $\gamma_i\neq 0$ and $\xi_i\neq 0$. Then~(\ref{4mp1}) implies $\gamma_i\geq 1/6$
and $\xi_i\geq 1/6$ which combined with
the upper bound of Claim~\ref{clm:prodbd} yields~(\ref{4mp2})
$$
(f(\gamma,\xi))_i \leq \frac{1}{3}\left(1-\frac{1/36}{1/3}\right)=\frac{11}{36}.
$$
Now we show~(\ref{4mp3}). Assume $f(\gamma,\xi)_i = f(\gamma,\xi)_j = 1/3$ for $i\neq j$. Then we have
$(\gamma_i=0 \vee \xi_i=0)$ and $(\gamma_j=0 \vee \xi_j=0)$. Note that at most one entry of $\gamma$
and at most one entry of $\xi$ can be $0$ (and then $\gamma,\xi\in\S_2$).
We can, w.l.o.g, assume $\gamma=(0,1/3,1/3,1/3)$ and $\xi=(1/3,0,1/3,1/3)$. Hence $f(\gamma,\xi)=(1/6,1/6,1/3,1/3)$.
\end{proof}

Lemma~\ref{lem:contract1} will follow from the following contraction
properties of~(\ref{eq:mproc}).

\begin{lemma}\label{lem:ineqb4}
Let $\alpha,\beta \in \S'_1$ be coupled, and let $\gamma \in \S'_1 \cup \S_2$, we have
$$
\|f(\alpha,\gamma)-f(\beta,\gamma)\|_{1}\leq \frac{1}{9 \sum_{i=1}^4 \alpha_i\beta_i} \|\alpha-\beta\|_{1}.
$$
\end{lemma}

\begin{proof}
Note that $\S'_1$ defined by equations~(\ref{4mp1})--(\ref{4mp3}) is not a convex set. However, if $\alpha,\beta \in \S'_1$ and $\alpha,\beta$ are coupled, then $(1-t)\alpha+t\beta \in \S'_1$ and $\alpha,(1-t)\alpha+t\beta$ are coupled. The lemma then follows from the same proof of Lemma~\ref{lem:ineqb}.
\end{proof}

\begin{lemma}\label{lem:sineq3}
Let $\alpha,\beta \in \S_1'$ be coupled, and let $\gamma \in \S_1'$ be
such that $\gamma$ has at most one entry of value $1/3$. Then we have
\begin{equation*}
\frac{49}{24}\|f(\alpha,\gamma)-f(\beta,\gamma)\|_1 \leq \|\alpha-\beta\|_1.
\end{equation*}
\end{lemma}

We first show how Lemma~\ref{lem:contract1} follows from Lemma~\ref{lem:ineqb4} and Lemma~\ref{lem:sineq3} and then prove Lemma~\ref{lem:sineq3}.
\begin{proof}[Proof of Lemma~\ref{lem:contract1}]
If $w \in U$ then from the assumption $\mathrm{dist}(w,\Delta) \geq 3$ we have
$\sigma_U(w)=\phi_U(w)$ and hence $\zeta = \eta$.
From now on we assume that $w \not\in U$ and thus $\zeta=f(\alpha,\alpha')$
and $\eta=f(\beta,\beta')$.

We will now show that $\alpha$ and $\beta$ are coupled. If $u \in U$, we have $\alpha=\beta$ (this follows from
$\sigma_U(u)=\phi_U(u)$, which is true since $\mathrm{dist}(u,\Delta)\geq 2$). Now assume $u \not\in U$.
Suppose $\alpha_i=1/3$ for  $i \in [4]$. By the definition of $f$ in~(\ref{eq:mproc}), we know that at least one child, say $v$,
of $u$ has color $i$ in $\sigma_U$. Note that $\mathrm{dist}(v,\Delta)\geq 1$ and hence
$\sigma_U(v)=\phi_U(v)$ which implies $\beta_i=1/3$. Hence $\alpha$ and $\beta$ are coupled.
The same argument yields that $\alpha'$ and $\beta'$ are coupled.

If $\alpha'=\beta'$ and $\alpha \neq \beta$, then $\alpha,\beta \in \S'_1$. Hence~(\ref{ehe}) follows from Lemma~\ref{lem:ineqb4} and Lemma~\ref{lem:prod} as we have
$$
\|f(\alpha,\alpha')-f(\beta,\alpha')\|_1 \leq \frac{1}{2}\|\alpha-\beta\|_1 \leq \frac{48}{49}\|\alpha-\beta\|_1 \leq \frac{48}{49}\max\{\|\alpha-\beta\|_1,\|\alpha'-\beta'\|_1\}.
$$
The same argument applies if $\alpha=\beta$ and $\alpha'\neq \beta'$, and hence from now on we assume $\alpha \neq \beta$ and $\alpha' \neq \beta'$.

We next claim that if one of $\alpha,\beta$ has two or more entries of value $1/3$, then $\alpha=\beta$.
By the previous paragraph, $\alpha$ and $\beta$ have value $1/3$ in the same entries (they are coupled). By~(\ref{4mp3})
they are either permutations of $(0,1/3,1/3,1/3)$ or $(1/6,1/6,1/3,1/3)$, and in both cases we have
$\alpha=\beta$ (using the fact that $\alpha$ and $\beta$ have value $1/3$ in the same entries).
The same argument applies to $\alpha'$ and~$\beta'$.

Now we can assume that each of $\alpha,\beta$ has most one entry of $1/3$ (otherwise, by the previous paragraph,
$\alpha=\beta$, a case that we already covered). Similarly each of $\alpha',\beta'$ has most one entry of $1/3$. Using triangle
inequality and Lemma~\ref{lem:sineq3} we obtain
\begin{equation*}
\begin{split}
 \|f(\alpha,\alpha')-f(\beta,\beta')\|_1 \leq \|f(\alpha,\alpha')-f(\alpha,\beta')\|_1 + \|f(\beta,\beta')-f(\alpha,\beta')\|_1 \leq\\
\frac{24}{49}\|\alpha-\beta\|_1 + \frac{24}{49}\|\alpha'-\beta'\|_1 \leq \frac{48}{49} \max \{ \|\alpha-\beta\|_1, \|\alpha'-\beta'\|_1 \}.
\end{split}
\end{equation*}
\end{proof}

Before we prove Lemma~\ref{lem:sineq3}, we need the following strengthening of Lemma~\ref{lem:prod}.
\begin{lemma}\label{lem:prodlb4}
Let $\gamma,\xi \in \S_1' \cup \S_2$. Then
either
\begin{equation}\label{eq:llb1}
\sum_{i=1}^4 \gamma_i \xi_i = 2/9,
\end{equation}
or $\sum_{i=1}^4 \gamma_i \xi_i \geq 49/216 > 2/9$,
where~(\ref{eq:llb1}) is attained only when
\begin{itemize}
\item $\gamma=(0,1/3,1/3,1/3)^\pi$ and $\xi=(1/3,\xi_2,\xi_3,\xi_4)^\pi$, or
\item $\xi=(0,1/3,1/3,1/3)^\pi$ and $\gamma=(1/3,\gamma_2,\gamma_3,\gamma_4)^\pi$, or
\item $\gamma=(1/6,1/6,1/3,1/3)^\pi$ and $\xi=(1/3,1/3,1/6,1/6)^\pi$,
\end{itemize}
where $\pi$ is a permutation of~$[4]$.
\end{lemma}

\begin{proof}
There are three cases depending on the numbers of $1/3$ in $\gamma$ and $\xi$.
\begin{itemize}
\item Case: $\gamma\in\S_2$ or $\xi\in\S_2$. We assume, w.l.o.g., $\gamma=(0,1/3,1/3,1/3)$. We have
$$
\sum_{i=1}^4 \gamma_i \xi_i = (1-\xi_1)/3 \geq 2/9,
$$
where the last inequality is attained only when $\xi_1=1/3$.
If $\sum_{i=1}^4 \gamma_i \xi_i \neq 2/9$, we have $\xi_1 \neq 1/3$, and by~(\ref{4mp2}) we have $\xi_1\leq 11/36$.
Hence
\begin{equation}\label{eq:elb1}
\sum_{i=1}^4 \gamma_i \xi_i = (1-\xi_1)/3 \geq 25/108 > 49/216.
\end{equation}
\item Case: $\gamma=(1/6,1/6,1/3,1/3)^\pi$ and $\xi \in \S_1'$, or $\xi=(1/6,1/6,1/3,1/3)^\pi$ and $\gamma \in \S_1'$,
where $\pi$ is a permutation of~$[4]$. We assume, w.l.o.g., $\gamma=(1/6,1/6,1/3,1/3)$ and $\xi \in \S_1'$. We have
$$
\sum_{i=1}^4 \gamma_i \xi_i = (\xi_1+\xi_2)/6+(\xi_3+\xi_4)/3 = 1/3 - (\xi_1+\xi_2)/6 \geq 2/9,
$$
where the last inequality is attained only when $\xi_1=\xi_2=1/3$. If $\sum_{i=1}^4 \gamma_i \xi_i \neq 2/9$, we have $\xi_1 \neq 1/3$ or $\xi_2 \neq 1/3$. By the definition of $\S_1'$, we have
\begin{equation}\label{eq:elb2}
\sum_{i=1}^4 \gamma_i \xi_i = 1/3 - (\xi_1+\xi_2)/6 \geq \frac{1}{3}-\frac{1}{6}\cdot\left(\frac{11}{36}+\frac{1}{3}\right)=\frac{49}{216}.
\end{equation}
\item Case: both $\gamma$ and $\xi$ have at most one entry of $1/3$. We
assume, w.l.o.g., $\gamma_1 \leq \gamma_2 \leq \gamma_3 \leq \gamma_4$.
Then the minimum of $\sum_{i=1}^4 \gamma_i \xi_i$ over $\xi\in\S'_1$ is achieved for $\xi=(1/3,11/36,7/36,1/6)$,
where the first entry is made as big as possible, the second entry is made
as big as possible (subject to~(\ref{4mp2})), and the last entry is made as small
as possible (subject to~(\ref{4mp1})). We have
\begin{equation}\label{eq:elb3}
\sum_{i=1}^4 \gamma_i \xi_i \geq \frac{\gamma_1}{3}+\frac{11\gamma_2}{36}+\frac{7\gamma_3}{36}+\frac{\gamma_4}{6} \geq \frac{1}{3} \cdot \frac{1}{6}+\frac{11}{36}\cdot\frac{7}{36}+\frac{7}{36}\cdot\frac{11}{36}+\frac{1}{6}\cdot\frac{1}{3} = \frac{149}{648} > \frac{49}{216}.
\end{equation}
\end{itemize}
The claim then follows from~(\ref{eq:elb1}),~(\ref{eq:elb2}) and~(\ref{eq:elb3}).
\end{proof}

We next prove Lemma~\ref{lem:sineq3}.

\begin{proof}[Proof of Lemma~\ref{lem:sineq3}]
Lemma~\ref{lem:sineq3} follows from Lemma~\ref{lem:ineqb4} and Lemma~\ref{lem:prodlb4}.
\end{proof}

\bibliographystyle{plain}
\bibliography{bibfile}

\end{document}